\documentclass{amsart}

\calclayout
\usepackage[english]{babel}
\usepackage[utf8x]{inputenc}
\usepackage[T1]{fontenc}

\usepackage[a4paper,top=3cm,bottom=2cm,left=3cm,right=3cm,marginparwidth=1.75cm]{geometry}
\usepackage{amsmath,amssymb,amsthm}
\usepackage{graphicx}
\usepackage[colorinlistoftodos]{todonotes}
\usepackage[colorlinks=true, allcolors=blue]{hyperref}

\newtheorem{theorem}{Theorem}[section]
\newtheorem{lemma}[theorem]{Lemma}
\theoremstyle{plain}                    
\newtheorem{teo}{Theorem}[section]      
\newtheorem{prop}[teo]{Proposition}    
\newtheorem{cor}[teo]{Corollary}       
    
\theoremstyle{definition}               
\newtheorem{defin}{Definition}[section]
\newtheorem{ese}{Example}[section]      
\theoremstyle{remark}                   
\newtheorem{oss}{Remark}[section]       
\newtheorem{que}{Question}[section]
\def\R{\mathbb{R}}
\def\H{\mathbb{H}}
\def\N{\mathbb{N}}

\title{Averages and the $\ell^{q,1}$-cohomology of Heisenberg groups}
\author{Pierre Pansu and Francesca Tripaldi}

\begin{document}
\begin{abstract}
    Averages are invariants defined on the $\ell^1$ cohomology of Lie groups. We prove that they vanish for abelian and Heisenberg groups. This result completes work by other authors and allows to show that the $\ell^1$ cohomology vanishes in these cases.
\end{abstract}
\maketitle

\section{Introduction}

\subsection{From isoperimetry to averages of $L^1$ forms}

The classical isoperimetric inequality implies that if $u$ is a compactly supported function on $\R^n$, $\|u\|_{n'}\leq C\|du\|_1$, where $n'=\frac{n}{n-1}$. Equivalently, every compactly supported closed $1$-form $\omega$ admits a primitive $u$ such that $\|u\|_{n'}\leq C\|\omega\|_1$. More generally, if $\omega$ is a closed $1$-form on $\R^n$ which belongs to $L^1$, does it have a primitive in $L^{n'}$ ? 

There is an obstruction. We observe that each component $a_i$ of $\omega=\sum_{i=1}^n a_i dx_i$ is again in $L^1$, the integral $\int_{\R^n}a_i \,dx_1\cdots dx_n$ is well defined and it is an obstruction for $\omega$ to be the differential of an $L^q$ function (for every finite $q$). Indeed, if $\omega=du$, $a_n=\frac{\partial u}{\partial x_n}$. For almost every $(x_1,\ldots,x_{n-1})$, the function $t\mapsto \frac{\partial u}{\partial x_n}(x_1,\ldots,x_{n-1},t)$ belongs to $L^1$ and $t\mapsto u(x_1,\ldots,x_{n-1},t)$ belongs to $L^q$. Since $u(x_1,\ldots,x_{n-1},t)$ tends to $0$ along subsequences tending to $+\infty$ or $-\infty$, 
$$
\int_{\R}\frac{\partial u}{\partial x_n}(x_1,\ldots,x_{n-1},t)\,dt=0, \quad\text{hence}\quad \int_{\R^n}\frac{\partial u}{\partial x_n}\,dx_1\cdots dx_n=0.
$$
A similar argument applies to other coordinates. Note that $a_n \,dx_1\wedge\cdots\wedge dx_n=(-1)^{n-1}\omega\wedge(dx_1\wedge\cdots\wedge dx_{n-1})$.

More generally, if $G$ is a Lie group of dimension $n$, there is a pairing, the \emph{average pairing}, between closed $L^1$ $k$-forms $\omega$ and closed left-invariant $(n-k)$-forms $\beta$, defined by
$$
(\omega,\beta)\mapsto \int_G \omega\wedge\beta.
$$
The integral vanishes if either $\omega=d\phi$ where $\phi\in L^1$, or $\beta=d\alpha$ where $\alpha$ is left-invariant. Indeed, Stokes formula $\int_{M}d\gamma=0$ holds for every complete Riemannian manifold $M$ and every $L^1$ form $\gamma$ such that $d\gamma\in L^1$. Hence the pairing descends to quotients, the $L^{1,1}$-cohomology 
$$
L^{1,1}H^k(G)=\text{closed }L^1 \,k\text{-forms}/d(L^1 \,(k-1)\text{-forms with differential in }L^1),
$$
and the Lie algebra cohomology
$$
H^{n-k}(\mathfrak{g})=\text{closed, left-invariant }(n-k)\text{-forms}/d(\text{left-invariant }(n-k-1)\text{-forms}).
$$

\subsection{$\ell^{q,1}$ cohomology}

It turns out that $L^{1,1}$-cohomology has a topological content. By definition, the $\ell^{q,p}$ cohomology of a bounded geometry Riemannian manifold is the $\ell^{q,p}$ cohomology of every bounded geometry simplicial complex quasiisometric to it. For instance, of a bounded geometry triangulation. Contractible Lie groups are examples of bounded geometry Riemannian manifolds for which $L^{1,1}$-cohomology is isomorphic to $\ell^{1,1}$-cohomology. 

We do not need define the $\ell^{q,p}$ cohomology of simplicial complexes here, since, according to Theorem 3.3 of \cite{Pansu-Rumin}, every $\ell^{q,p}$ cohomology class of a contractible Lie group can be represented by a form $\omega$ which belongs to $L^p$ as well as an arbitrary finite number of its derivatives. If the class vanishes, then there exists a primitive $\phi$ of $\omega$ which belongs to $L^q$ as well as an arbitrary finite number of its derivatives. This holds for all $1\leq p\leq q\leq\infty$. 

Although $\ell^p$ with $p>1$, and especially $\ell^2$ cohomology of Lie groups has been computed and used for large families of Lie groups, very little is known about $\ell^1$ cohomology.

\subsection{From $\ell^{1,1}$ to $\ell^{q,1}$ cohomology}

For instance, the averaging pairing is specific to $\ell^1$ cohomology and it has never been studied yet. The first question we want to address is whether the averaging pairing provides information on $\ell^{q,1}$ cohomology for certain $q>1$.

\begin{que}
Given a Lie group $G$, for which exponents $q$ and which degrees $k$ is the averaging pairing $\ell^{q,1}H^k(G)\otimes H^{n-k}(\mathfrak{g})\to\R$ well-defined?
\end{que}

The question is whether there exists $q>1$ such that the pairing vanishes on all $L^1$ forms which are differentials of $L^q$ forms. We just saw that for abelian groups $\R^n$, the pairing is well defined for $k=1$ and all finite exponents $q$. Here is a more general result.

\begin{theorem}
\label{defined}
Let $G$ be a Carnot group. In each degree $1\leq k\leq n$, there is an explicit exponent $q(G,k)>1$ (see Definition \ref{jq}) such that the averaging pairing is defined on $\ell^{q,1}H^k(G)$ for $q\in[1,q(G,k)]$.
\end{theorem}
We shall see that $q(\R^n,k)=n'=\frac{n}{n-1}$ in all degrees. For Heisenberg groups, $q(\H^{2m+1},k)=\frac{2m+2}{2m+1}$ if $k\not=m+1$, and $q(\H^{2m+1},m+1)=\frac{2m+2}{2m}$.

\subsection{Vanishing of the averaging pairing}

The second question we want to address is whether the averaging pairing is trivial or not.

\begin{que}
Given a Lie group $G$, for which exponents $q$ and which degrees $k$ does the averaging pairing $\ell^{q,1}H^k(G)\otimes H^{n-k}(\mathfrak{g})\to\R$ vanish?
\end{que}
The pairing is always nonzero in top degree $k=n$. Indeed, there exist $L^1$ $n$-forms (even compactly supported ones) with nonvanishing integral. However, this seems not to be the case in lower degrees.

\begin{theorem}
\label{vanish}
Let $G$ be an abelian group or a Heisenberg group of dimension $n$. In each degree $1\leq k< n$, the averaging pairing vanishes on $\ell^{q,1}H^k(G)$ for $q\in[1,q(G,k)]$.
\end{theorem}

In combination with results of \cite{BFP3}, Theorem \ref{vanish} implies a vanishing theorem for $\ell^{q,1}$ cohomology.

\begin{cor}
Let $G$ be an abelian group or a Heisenberg group of dimension $n$. In each degree $0\leq k< n$, $\ell^{q,1}H^k(G)=0$ for $q\geq q(G,k)$.
\end{cor}

This is sharp. It is shown in \cite{Pansu-Rumin} that $\ell^{q,1}H^k(G)\not=0$ if $q<q(G,k)$. Also, in top degree, not only is $\ell^{q(G,n),1}H^n(G)\not=0$, but the kernel of the averaging map $\ell^{q(G,n),1}H^n(G)\to\R=H^0(\mathfrak{g})^*$ does not vanish. This is in contrast with the results of \cite{BFP2} concerning $\ell^{q,p}H^n(G)$ for $p>1$, where nothing special happens in top degree. The results of \cite{BFP3} rely in an essential manner on analysis of the Laplacian on $L^1$, inaugurated by J. Bourgain and H. Brezis, \cite{Bourgain2007}, adapted to homogeneous groups by S. Chanillo and J. van Schaftingen, \cite{chanillo2009subelliptic}.

\subsection{Methods}

The Euclidean space $\R^n$ is \emph{$n$-parabolic}, meaning that there exist smooth compactly supported functions $\xi$ on $\R^n$ taking value $1$ on arbitrarily large balls, and whose gradient has an arbitrarily small $L^n$ norm. If $\omega$ is a closed $L^1$ form and $\beta$ a constant coefficient form, and if $\omega=d\psi$, $\psi\in L^{n'}$, Stokes theorem gives 
$$
|\int\xi\omega\wedge\beta|=|\int\psi\wedge d\xi\wedge\beta|\leq \|\psi\|_{n'}\|d\xi\|_{n}\|\beta\|_\infty
$$
which can be made arbitrarily small.

This argument extends to Carnot groups of homogeneous dimension $Q$, which are $Q$-parabolic. For this, one uses Rumin's complex, which has better homogeneity properties under Carnot dilations than de Rham's complex. When Rumin's complex is exactly homogeneous (e.g. for Heisenberg groups in all degrees, only for certain degrees in general), one gets a sharp exponent $q(G,k)$. This leads to Theorem \ref{defined}.

In Euclidean space, every constant coefficient form $\beta$ has a primitive $\alpha$ with linear coefficients (for instance, $dx_1\wedge\cdots\wedge dx_k=d(x_1\,dx_2\wedge\cdots\wedge dx_k)$). On the other hand, there exist cut-offs which decay like the inverse of the distance to the origin. Therefore
$$
|\int\xi\omega\wedge\beta|=|\int\omega\wedge d\xi\wedge\alpha|\leq \|\omega\|_{L^1(\text{supp}(d\xi))}
$$
which tends to $0$. This argument extends to Heisenberg groups in all but one degree. To complete the proof of Theorem \ref{vanish}, one performs the symmetric integration by parts, integrating $\omega$ instead of $\beta$. For this, one produces primitives of $\omega$ on annuli, of linear growth.

\subsection{Organization of the paper}

In section \ref{cut-off}, the needed cut-offs are constructed. Theorem \ref{defined} is proven in section \ref{descends}. In order to integrate $\omega\wedge\beta$ by parts, one needs understand the behaviour of wedge products in Rumin's complex on Heisenberg groups, this is performed in section \ref{wedge}. Section \ref{primitives} exploits the linear growth primitives of left-invariant forms. In section \ref{poincare}, controlled primitives of $L^1$ forms are designed, completing the proof of Theorem \ref{vanish}.

\section{Cut-offs on Carnot groups}
\label{cut-off}

\subsection{Annuli}

We first construct cut-offs with an $L^\infty$ control on derivatives. In a Carnot group $G$, we fix a subRiemannian metric and denote by $B(R)$ the ball with center $e$ and radius $R$. We fix an orthonormal basis of horizontal left-invariant vector fields $W_1,\ldots,W_{n_1}$. Given a smooth function $u$ on $G$, and an integer $m\in\N$, we denote by $\nabla^m u$ the collection of order $m$ horizontal derivatives $W_{i_1}\cdots W_{i_m}$, $(i_1,\ldots,i_m)\in \{1,\ldots,n_1\}^m$, and by $|\nabla^m u|^2$ the sum of their squares.

\begin{lemma}
Let $G$ be a Carnot group. Let $\lambda>1$. There exists $C=C(\lambda)$ such that for all $R>0$, there exists a smooth function $\xi_R$ such that
\begin{enumerate}
  \item $\xi_R=1$ on $B(R)$.
  \item $\xi_R=0$ outside $B(\lambda R)$.
  \item For all $m\in\N$, $|\nabla^m\xi_R|\leq C/R^m$.
\end{enumerate}
\end{lemma}

\begin{proof}
We achieve this first when $R=1$, and then set $\xi_R=\xi_1\circ \delta_{1/R}$.
\end{proof}
\begin{lemma}\label{nabla}
Given $f$ a (vector valued) function which is homogeneous of degree $d\in\mathbb{N}$ under dilations, then $\nabla f$ is homogeneous of degree $d-1$.
\end{lemma}
\begin{proof}
Given $f\colon G\to\mathbb{R}$ homogeneous of degree $d$ under dilations, we have that $f(\delta_\lambda p)=\lambda^df(p)$.

By applying a horizontal derivative to the left hand side of the equation, namely $\nabla=W_{j}$ with $j\in\lbrace 1,\ldots,n_1\rbrace$, we get
\begin{align*}
    \nabla [f(\delta_\lambda p)]=W_j[f(\delta_\lambda p)]=df\circ d\delta_\lambda(W_j)_p=df\big(\lambda (W_j)_{\delta_\lambda p}\big)=\lambda\cdot df\big(W_j\big)_{\delta_\lambda p}\,.
\end{align*}

If we now apply $\nabla$ to the right hand side, we get
\begin{align*}
    \nabla[\lambda^df(p)]=W_j[\lambda^df(p)]=\lambda^d\cdot df(W_j)_p\,.
\end{align*}

We have therefore proved that $df\big\vert_{\delta_\lambda p}=\lambda^{d-1}\cdot df\big\vert_p$ when restricted to horizontal derivatives, so we finally get our result
\begin{align*}
    \nabla f(\delta_\lambda p)=\lambda^{d-1}\nabla f(p)\,.
\end{align*}
\end{proof}

\subsection{Parabolicity}

Second, we construct cut-offs with a sharper $L^Q$ control on derivatives.

Let $r$ be a smooth, positive function on $G\setminus\{e\}$ that is homogeneous of degree 1 under dilations (one could think of a CC-distance from the origin, but smooth) and let us define the following function
\begin{align}\label{chi function}
\chi(r)=\frac{\log(\lambda R/r)}{\log(\lambda R/R)}=\frac{\log(\lambda R/r)}{\log(\lambda )}\,.
\end{align}

One should notice that $\chi(\lambda R)=0$, $\chi(R)=1$, and that $\chi$ is smooth.

\begin{defin}\label{cutoff defin} Using the smooth function $\chi$ introduced in (\ref{chi function}), we can then define the cut-off function $\xi$ as follows
\begin{align*}
\xi(r)=\begin{cases}
    1,& \text{on } B(R)\\
    \chi(r),              & \text{on } B(\lambda R)\setminus B(R)\\ 0,& \text{outside } B(\lambda R)\,.
\end{cases}
\end{align*}
\end{defin}





\begin{lemma}\label{norms of derivative of xi}
The cut-off function $\xi$ defined above has the following property: for every integer $m\in\N$,
$\Vert\nabla^m\xi\Vert_{Q/m}\to 0$ as $\lambda\to\infty$.
\end{lemma}

\begin{proof}

We compute
\begin{align*}
\nabla\xi=\begin{cases}
\frac{1}{\log\lambda}\frac{\nabla r}{r}&\text{if }R<r<\lambda R,\\
0&\text{otherwise.}
\end{cases}
\end{align*}
Let $f$ be the vector valued function $f=\frac{\nabla r}{r}$ on $G\setminus\{e\}$. Then $f$ is homogeneous of degree $-1$. According to Lemma \ref{nabla}, $\nabla^{m-1}f$ is homogeneous of degree $m$. It follows that
\begin{align*}
\int_{B(\lambda R)\setminus B(R)}\bigg\vert \nabla^{m-1}f\bigg\vert^{Q/m}&\le C\int_R^{\lambda R}\bigg\vert\frac{1}{r^m}\bigg\vert^{Q/m} r^{Q-1}dr\\
=C&\int_R^{\lambda R}\frac{r^{Q-1}}{r^Q}dr=C\log(\lambda).
\end{align*}
Therefore
$$
\|\nabla^m\xi\|_{Q/m}\leq C\,(\log\lambda)^{-1+(m/Q)}
$$
tends to $0$ as $\lambda$ tends to infinity, provided $m<Q$.
Let us stress that, in general, for values of $m$ greater than or equal to $Q$, the final limit will not be zero. However, the values of $m$ which we will be considering are very specific.

This estimate will in fact be used in the proof of Proposition \ref{w-j}, and in that setting the degrees $m$ that can arise will be all the possible degrees (or equivalently weights) of the differential $d_c$ on an arbitrary Rumin $k$-form $\phi$ of weight $w$. If we denote by $M$ the maximal $m$ that could arise in this situation, then one can show that $M<Q$.

Let us first assume that the maximal order $M$ for the $d_c$ on $k$-forms (which is non trivial for $0\le k<n$) is greater or equal to $Q$. Then, given $\phi$ a Rumin $k$-form of weight $w$, then $d_c\phi=\sum_{i=1}^M\beta_{w+i}$, where each $\beta_{w+i}$ is a Rumin $(k+1)$-form of weight $w+i$. If we consider the Hodge of the $\beta_{w+i}$ with $Q\le i\le M$, then these forms are Rumin $(n-k-1)$-forms of weight $Q-(w+i)=Q-w-i\le Q-w-Q<-w\le 0$, which is impossible.

Therefore $M<Q$, which means that indeed in all the cases that we will take into consideration, the $L^{Q/m}$ norm of $\nabla^m\xi$ will always go to zero as $\lambda\to \infty$.

\end{proof}

\begin{oss}
One says that a Riemannian manifold $M$ is \emph{$p$-parabolic} (see \cite{Troyanov}) if for every compact set $K$, there exist smooth compactly supported functions on $M$ taking value $1$ on $K$ whose gradient has an arbitrarily small $L^p$ norm. The definition obviously extends to subRiemannian manifolds.

Lemma \ref{norms of derivative of xi} implies that a Carnot group of homogeneous dimension $Q$ is $Q$-parabolic.
\end{oss}

\section{The averaging map in general Carnot groups descends to cohomology}
\label{descends}

\begin{defin}
Let $G$ be a Carnot group of dimension $n$ and homogeneous dimension $Q$. For $k=1,\ldots,n$, let $\mathcal{W}(k)$ denote the set of weights arising in Rumin's complex in degree $k$. For a Rumin $k$-form $\omega$, let
$$
\omega=\sum_{w\in\mathcal{W}(k)} \omega_{w}
$$ be its decomposition into components of weight $w$.

Let $d_c=\sum_j d_{c,j}$ be the decomposition of $d_c$ into weights/orders. Let $\mathcal{J}(k,w)$ denote the set of weights/orders $j$ such that $d_{c,j}$ on $k$-forms of weight $w$ is nonzero, in other words,
\begin{align*}
    \mathcal{J}(k,w):=\lbrace j\in\mathbb{N}\mid d_{c,j}\omega_w\neq 0 \text{ for some }\omega\text{ of degree }k\rbrace\,.
\end{align*}

We will denote by $\mathcal{J}(k)$ the set of all the possible weights/orders, that is 
\begin{align*}
    \mathcal{J}(k)=\bigcup_{w\in \mathcal{W}(k)}\mathcal{J}(k,w)\,.
\end{align*}

Let us define
 $L^{\chi(k)}$ as follows
\begin{align*}
    L^{\chi(k)}=\lbrace \phi=\sum_{w\in\mathcal{W}(k-1)}\phi_w\in E_0^{k-1}\mid \forall j\in\mathcal{J}(k-1,w)\,,\phi_{w}\in L^{Q/Q-j}\rbrace
\end{align*}
and if $\mathcal{J}(k-1,\hat{w})=\emptyset$ for some $\hat{w}$, then we don't require anything on $\phi_{\hat{w}}$.
\end{defin}

\begin{lemma}
Let $G$ be a Carnot group. Fix a left-invariant subRiemannian metric making the direct sum $\mathfrak{g}=\bigoplus \mathfrak{g}_i$ orthogonal. The $L^2$-adjoint $d_c^*$ of $d_c$ is a differential operator. Fix a degree $k$. Let 
$$
d_c=\sum_{j\in\mathcal{J}(k)} d_{c,j}
$$
be the decomposition of $d_c$ into weights ($d_{c,j}$ increases weights of Rumin forms by $j$, hence it has horizontal order $j$). Then the decomposition of $d_c^*$ on a $(k+1)$-form into weights/orders is
$$
d_c^*=\sum_{j\in\mathcal{J}^\ast(k)} d_{c,j}^*.
$$

In other words, the adjoint $d_{c,j}^*$ of $d_{c,j}$ decreases weights by $j$ and has horizontal order $j$. In fact, if we denote by $\mathcal{J}^*(k+1,\tilde{w})$ the set of weights/orders $j$ such that $d_{c,j}^*$ on $(k+1)$-forms of weight $\tilde{w}$ is non-zero, that is
\begin{align*}
    \mathcal{J}^\ast(k+1,\tilde{w})=\lbrace j\in\mathbb{N}\mid d_{c,j}^*\alpha_{\tilde{w}}\neq 0\text{ for some }\alpha\text{ of degree }k+1\rbrace\,,
\end{align*}
then there is a clear relationship between the sets of indices $\mathcal{J}(k,w)$ and $\mathcal{J}^*(k+1,\tilde{w})$, namely
\begin{align*}
    \mathcal{J}^\ast(k+1,\tilde{w})=\bigcup_{w\in\mathcal{W}(k)}\lbrace j\in\mathcal{J}(k,w)\mid w+j=\tilde{w}\rbrace\,.
\end{align*}

And from this relationship, we get directly the following identity:
\begin{align*}
    \mathcal{J}^*(k+1)=\bigcup_{\tilde{w}\in\mathcal{W}(k+1)}\mathcal{J}^*(k+1,\tilde{w})=\mathcal{J}(k)\,.
\end{align*}

Moreover, since the formula $d_c^\ast=(-1)^{n(k+1)+1}\ast d_c\ast$ applies to any Rumin $k$-form, we also have the equality $\mathcal{J}^\ast(n-k,Q-w)=\mathcal{J}(k,w)$.

{Let us stress that this also implies $\mathcal{J}(k)=\mathcal{J}^\ast(n-k)=\mathcal{J}(n-k-1)$}. 

\end{lemma}

\begin{prop}
\label{w-j}
If $\omega$, $\phi$, $\beta$ are Rumin forms with $\omega\in L^1$ of degree $k$, $\beta$ left-invariant of complementary degree $n-k$, $\phi\in L^{\chi(k)}$ and $d_c \phi=\omega$, then
$$
\int_G \omega\wedge\beta=0.
$$
\end{prop}

\textbf{Proof}. Without loss of generality, one can assume that $\beta$ has pure weight $Q-w$ for some $w\in\mathcal{W}(k)$. Then its Hodge-star $\ast\beta$ has pure weight $w$. Let $\xi$ be a smooth cut-off. By definition of the $L^2$-adjoint, we have
$$
\int_{G}(d_c\phi)\wedge\xi\beta=\int_{G}\langle d_c\phi ,\ast\xi\beta\rangle\,dvol=\int_{G}\langle\phi,d_c^*(\ast\xi\beta)\rangle\,dvol=\sum_{j\in\mathcal{J}^*(k,w)}\int_{G}\langle\phi_{w-j},d_{c,j}^*(\ast\xi\beta)\rangle\,dvol.
$$

Since $\phi\in L^{\chi(k)}$, for any $j\in\mathcal{J}^\ast(w-j)$ we have $\phi_{w-j}\in L^{Q/Q-j}$, by definition of $L^{\chi(k)}$. Hence, applying H\"older's inequality, we obtain
$$
|\int_{G}\omega\wedge\xi\beta|\leq \sum_{j\in\mathcal{J}^*(k,w)}\|\phi_{w-j}\|_{Q/(Q-j)}\|\nabla^j\xi\|_{Q/j}\|\beta\|_{\infty}.
$$

It is therefore sufficient to take $\xi$ as the cut-off function introduced in Definition \ref{cutoff defin}, so that by Lemma \ref{norms of derivative of xi} we get that $\int_{G}\omega\wedge\beta=0$.

\begin{ese}
Euclidean space $\R^n$. Then $\mathcal{W}(k)=\{k\}$ and $\mathcal{J}(k)=\{1\}$ in all degrees. Theorem \ref{w-j} states that the averaging map descends to $L^{q,1}$-cohomology, where $q=\frac{n}{n-1}$.
\end{ese}

\begin{ese}
Heisenberg groups $\mathbb{H}^{2m+1}$. We have that $\mathcal{W}(k)=\{k\}$ for $k\le m$, and $\mathcal{W}(k)=\{k+1\}$ when $k\geq m+1$. $\mathcal{J}(k)=\{1\}$ in all degrees but $k=m$, where $\mathcal{J}(m)=\lbrace 2\rbrace$, so that Theorem \ref{w-j} states that the averaging map descends to $L^{q,1}$-cohomology, where $q=\frac{Q}{Q-2}$ in degree $m+1$ and $q=\frac{Q}{Q-1}$ in all other degrees.
\end{ese}

\subsection{Link with $\ell^{q,1}$ cohomology}

Let $1\leq p\leq q\leq\infty$. According to Theorem 3.3 of \cite{Pansu-Rumin}, every $\ell^{q,p}$ cohomology class of a Carnot group contains a form $\omega$ which belongs to $L^p$ as well as an arbitrary finite number of its derivatives. If the class vanishes, then there exists a primitive $\phi$ of $\omega$ which belongs to $L^q$ as well as an arbitrary finite number of its derivatives. There exists a homotopy between de Rham and Rumin's complex given by differential operators, therefore the same statement applies to Rumin's complex. In particular, Rumin forms can be used to compute $\ell^{q,p}$ cohomology.

Let $\omega$ be a Rumin $k$-form which belongs to $L^1$ as well as a large number of its  horizontal derivatives. Assume that $\omega$ represents the trivial cohomology class. Then there exists a Rumin $(k-1)$-form $\phi$ which belongs to $L^q$ as well as its horizontal derivatives up to order $Q$, and such that $d_c\phi=\omega$. By Sobolev's embedding theorem, $\phi$ belongs to $L^\infty$, hence to $L^{q'}$ for all $q'\geq q$. This suggests the following notation.

\begin{defin}
\label{jq}
Let $G$ be a Carnot group of dimension $n$ and homogeneous dimension $Q$. Let $1\leq k\leq n$. Define
$$
j(k)=\min\bigcup_{w\in \mathcal{W}(k)} \mathcal{J}(k-1,w).
$$
and
$$
q(G,k):=\frac{Q}{Q-j(k)}.
$$
\end{defin}

\textbf{Proof of Theorem \ref{defined}}.
\begin{proof}

Let $G$ be a Carnot group. Let $\omega$ be a Rumin $k$-form on $G$, which belongs to $L^1$ as well as a large number of its derivatives. Assume that $\omega=d_c\phi$ with $\phi\in L^{q(G,k)}$. Then  
$$
\forall w\in \mathcal{W}(k),\,\forall j\in\mathcal{J}(k-1,w),\quad \phi_{w-j} \in L^{Q/(Q-j)},
$$
therefore $\phi\in L^{\chi(k)}$. Proposition \ref{w-j} implies that averages $\int\omega\wedge\beta$ vanish. This completes the proof of Theorem \ref{defined}.

\end{proof}
\begin{ese}
Euclidean space. Then $j(k)=1$ in all degrees.
\end{ese}

\begin{ese}
Heisenberg groups $\mathbb{H}^{2m+1}$. Then $j(k)=1$ in all degrees but $k=m+1$, where $j(m+1)=2$.
\end{ese}
For these examples, as we saw before, one need not invoke \cite{Pansu-Rumin} since $\mathcal{J}(k)$ has only one element in each degree.

\begin{ese}
Engel group $E^4$. Then $j(k)=1$ in degrees $1$ and $4$, $j(k)=2$ in degrees $2$ and $3$. One concludes that the averaging map is well-defined in $\ell^{q,1}$ cohomology for $q\leq\frac{Q}{Q-1}$ in degrees $1$ and $4$, and for $q\leq\frac{Q}{Q-2}$ in degrees $2$ and $3$. Here, $Q=7$.
\end{ese}

\subsection{Results for Heisenberg groups $\mathbb{H}^{2m+1}$}
In \cite{BFP3}, it is proven that every closed $L^1$ $k$-form, $k\leq 2m$, whose averages $\int\omega\wedge \beta$ vanish, is the differential of a form in $L^q$, where $q=q(k)=\frac{Q}{Q-1}$ unless when $k=m+1$, where $q(m+1)=\frac{Q}{Q-2}$. In other words,

\begin{theorem}[\cite{BFP3}]
Let $G=\mathbb{H}^{2m+1}$ and let $k=1,\ldots,2m$. The averaging map $L^{q(k),1}H^k(G)\to H^{2m+1-k}(\mathfrak{g})^*$ is injective.
\end{theorem}

The goal of subsequent sections is to prove that the image of averaging map is $0$ in all degrees $k\leq 2m$. This will prove that $L^{q(k),1}H^k(G)=0$. Note that for $k=2m+1$, both facts fail: the averaging map is not zero (one can check with compactly supported forms) and it is not injective either (see \cite{BFP3}).

\section{Wedge products between Rumin forms in Heisenberg groups}
\label{wedge}

We shall rely on Stokes formula on Heisenberg groups $\mathbb{H}^{2m+1}$. We need a formula of the form $d(\phi\wedge\beta)=(d_c\phi)\wedge\beta\pm \phi\wedge d_c\beta$. This does not always hold in general for Carnot groups. In fact, the complex of Rumin forms $ E_0^\bullet$ equals the Lie algebra cohomology $H^\bullet(\mathfrak{g})$, and therefore carries a natural cup product induced by the wedge product, but which in general differs from the wedge product. 

Let us take into consideration the original construction of the Rumin complex in the $(2m+1)$-dimensional Heisenberg group $\mathbb{H}^{2m+1}$ as appears in \cite{rumin1994}.

Given $\Omega^\bullet$ the algebra of smooth differential forms, one can define the following two differential ideals:
\begin{itemize}
\item $\mathcal{I}^\bullet:=\lbrace \alpha=\gamma_1\wedge\tau+\gamma_2\wedge d\tau\rbrace$, the differential ideal generated by the contact form $\tau$, and
\item $\mathcal{J}^\bullet:=\lbrace\beta\in\Omega^\bullet\mid\beta\wedge\tau=\beta\wedge d\tau=0\rbrace$.
\end{itemize}

\begin{oss}\label{wedge between ideals}
By construction, the ideal $\mathcal{J}^\bullet$ is in fact the annihilator of $\mathcal{I}^\bullet$. In other words, given two arbitrary forms $\alpha\in\mathcal{J}^\bullet$ and $\beta\in\mathcal{I}^\bullet$, we have $\alpha\wedge\beta$=0.
\end{oss}

One can quickly check that the subspaces $\mathcal{J}^h=\mathcal{J}^\bullet\cap\Omega^h$ are non-trivial for $h\ge m+1$, whereas the quotients $\Omega^h/\mathcal{I}^h$ are non-trivial for $h\le m$, where $\mathcal{I}^h=\mathcal{I}^\bullet\cap\Omega^h$.

Moreover, the usual exterior differential descends to the quotients $\Omega^\bullet/\mathcal{I}^\bullet$ and restricts to the subspaces $J^\bullet$ as first order differential operators:
\begin{align*}
d_c:\Omega^\bullet/\mathcal{I}^\bullet\to\Omega^\bullet/\mathcal{I}^\bullet\;\text{ and }\;d_c:\mathcal{J}^\bullet\to\mathcal{J}^\bullet\,.
\end{align*}

In \cite{rumin1994} Rumin then defines a second order linear differential operator
\begin{align*}
d_c:\Omega^m/\mathcal{I}^m\to\mathcal{J}^{m+1}
\end{align*}
which connects the non-trivial quotients $\Omega^\bullet/\mathcal{I}^\bullet$ with the non-trivial subspaces $\mathcal{J}^\bullet$ into a complex, that is $d_c\circ d_c=0$,
\begin{align*}
\Omega^0/\mathcal{I}^0\xrightarrow{d_c}\Omega^1/\mathcal{I}^1\xrightarrow{d_c}\cdots\xrightarrow{d_c}\Omega^m/\mathcal{I}^m\xrightarrow{d_c}\mathcal{J}^{m+1}\xrightarrow{d_c}\mathcal{J}^{m+2}\xrightarrow{d_c}\cdots\xrightarrow{d_c}\mathcal{J}^{2m+1}\,.
\end{align*}

\begin{prop}
\label{dc of Wedge Product}
In $\mathbb{H}^{2m+1}$, the wedge product of Rumin forms is well-defined and satisfies the Leibniz rule 
\begin{align*}
d_c(\alpha\wedge\beta)=d_c\alpha\wedge\beta+(-1)^h\alpha\wedge d_c\beta
\end{align*}
if either $h\ge m+1$ or $k\ge m+1$ or $h+k< m$, where $h=deg(\alpha)$ and $k=deg(\beta)$.
\end{prop}

\begin{proof}
In order to study whether the wedge product between Rumin forms is well-defined, we will consider this differential operator $d_c$ in the following two cases:
\begin{itemize}
\item [i.] $d_c:\Omega^h/\mathcal{I}^h\to\Omega^{h+1}/\mathcal{I}^{h+1}$ where $h <m$,
\item[ii.] $d_c:\mathcal{J}^h\to\mathcal{J}^{h+1}$ where $h>m$.
\end{itemize}
Let us first stress that in the first case, given $\alpha\in\Omega^h/\mathcal{I}^h$, we have that 
\begin{align*}
d_c\alpha=d\alpha  \,\mod\;\mathcal{I}^{h+1}\;\text{ for }\;h<m\,.
\end{align*}
Since $\mathcal{I}$ is an ideal, if $h+k\leq m$, $\alpha\wedge\beta\in \Omega^{h+k}/\mathcal{I}^{h+k}$ is well defined. 

If $h+k<m$, the identity $d(\alpha\wedge\beta)=(d\alpha)\wedge\beta+(-1)^h\alpha\wedge d\beta$ passes to the quotient. 

It is important to notice that, however, if $h+k=m$, $h>0$, $k>0$, $d_c\alpha\wedge\beta+(-1)^h\alpha\wedge d_c\beta$ involves only first derivatives of $\alpha$ and $\beta$, and thus cannot be equal to $d_c(\alpha\wedge\beta)$. If $h=0$ and $k=m$, $d_c(\alpha\wedge\beta)$ involves second derivatives of $\alpha$, and therefore cannot be expressed in terms of $d_c\alpha$.

On the other hand, in the second case, given $\beta\in\mathcal{J}^h$, the Rumin differential coincides with the usual exterior differential,
\begin{align*}
d_c\beta=d\beta\;\text{ for }\;h>m\,.
\end{align*}

Therefore, given $\alpha\in\Omega^{h}/\mathcal{I}^{h}$ and $\beta\in\mathcal{J}^{k}$ with $h<m$ and $k>m$, the wedge product $\alpha\wedge\beta$ is well-defined and belongs to $\mathcal{J}^{h+k}$, and the usual Leibniz rule also applies:
\begin{align*}
d_c(\alpha\wedge\beta)=d(\alpha\wedge\beta)=(d\alpha)\wedge\beta+(-1)^h\alpha\wedge (d\beta)=d_c\alpha\wedge\beta+(-1)^h\alpha\wedge d_c\beta\,.
\end{align*}
If $h=m$ and $k\geq m+1$, $h+k\geq 2m+1$, so the identity between differentials holds trivially.

To conclude, the wedge product of Rumin forms is well defined and satisfies the Leibniz rule $d_c(\alpha\wedge\beta)=d_c\alpha\wedge\beta+(-1)^h\alpha\wedge d_c\beta$ if either $h\ge m+1$, or $k\ge m+1$, or $h+k<m$.

\end{proof}

\section{Averages on Heisenberg group: generic case}
\label{primitives}

\subsection{Primitives of linear growth}

\begin{lemma}
Let $\beta$ be a left-invariant Rumin $h$-form in the Heisenberg group $\mathbb{H}^{2m+1}$. If $h\not=m+1$, $\beta$ admits a primitive $\alpha$ of linear growth, i.e. at Carnot-Carath\'eodory distance $r$ from the origin, $|\alpha|\leq C\,r$. 
\end{lemma}

\begin{proof}
Let $\beta\in E_0^{h}$ be a left-invariant form. Then $d_c\beta=0$, and $\beta$ has weight $w=h$ (if $h\leq m$) or $h+1$ (if $h>m+1$). We know that the Rumin complex is locally exact, that is $\exists\, \alpha\in E_0^{h-1}$ such that $d_c\alpha=\beta$.

Let us consider the Taylor expansion of $\alpha$ at the origin in exponential coordinates, and let us group terms according to their homogeneity under dilations $\delta_t$:
\begin{align*}
    \alpha=\alpha_0+\cdots+\alpha_{w-1}+\alpha_w+\alpha_{w+1}+\cdots\,
\end{align*}
where we denote by $\alpha_d$ the term with homogeneous degree d, i.e. $\delta_t^\ast\alpha_d=t^d\alpha_d$.

Since $d_c$ commutes with the dilations $\delta_\lambda$, the expansion of $d_c\alpha$ is therefore 
\begin{align*}
    d_c\alpha=d_c\alpha_0+\cdots+d_c\alpha_{w-1}+d_c\alpha_w+d_c\alpha_{w+1}+\cdots\,.
\end{align*}

The expansion of $\beta$ is given instead by $\beta=\beta$, given it is a left-invariant form, hence homogeneous of degree $w$, so that $\beta=d_c\alpha_w$.

Let us notice that $\alpha_w$ has degree $h-1$ and $h\not=m+1$, so it has weight $w-1$. Since it is homogeneous of degree $w$ under $\delta_\lambda$, its coefficients are homogeneous of degree $1$, that is they are linear in horizontal coordinates, hence $\alpha_w$ has linear growth, that is {$\vert\alpha\vert\le {C}\,{r}$}.
\end{proof}

\begin{prop}
Given $\omega\in E_0^{k}$ an $L^{1}$, $d_c$-closed Rumin form in $\mathbb{H}^{2m+1}$, then the integral
\begin{align*}
\int_{\mathbb{H}^{2m+1}}\omega\wedge\beta
\end{align*}
vanishes for all left-invariant Rumin forms $\beta$ of complementary degree, $\beta\in E_0^{2m+1-k}$, provided $k\not=m$.
\end{prop}

\begin{proof}
Let $\omega$ be an $L^1$, $d_c$-closed Rumin $k$-form, $k\not=m$. Let $\beta$ be a left-invariant Rumin $h$-form, with $h=2m+1-k\neq m+1$. Let $\alpha$ be a linear growth primitive of $\beta$, $\vert\alpha\vert\le {C}\,{R}$. Let $\xi$ be a smooth cut-off such that $\xi=1$ on $B(R)$, $\xi=0$ outside $B(\lambda R)$ and $|d_c\xi|\leq C'/R$. Since, according to Proposition \ref{dc of Wedge Product},
\begin{align*}
    d(\xi\omega\wedge\alpha)=d_c(\xi\omega\wedge\alpha)
    &= d_c(\xi\omega)\wedge\alpha+(-1)^{k}\xi\omega\wedge d_c\alpha\\
    &=d_c\xi\wedge\omega\wedge\alpha+(-1)^{k}\xi\omega\wedge\beta,
\end{align*}{}
Stokes formula gives
\begin{align*}
\bigg\vert\int_{\H^{2m+1}}\xi\omega\wedge\beta\bigg\vert
&=\bigg\vert\int_{B(\lambda R)\setminus B(R)}d_c\xi\wedge\omega\wedge\alpha\bigg\vert\\
&\le\int_{B(\lambda R)\setminus B(R)}\vert d_c\xi\vert\vert\alpha\vert\vert\omega\vert\\ 
&\le \lambda CC'\Vert\omega\Vert_{L^1(\H^{2m+1}\setminus B(R))}.
\end{align*}
On the other hand,
\begin{align*}
\bigg\vert\int_{\H^{2m+1}}(1-\xi)\omega\wedge\beta\bigg\vert
&\le \Vert \beta\Vert_\infty \Vert\omega\Vert_{L^1(\H^{2m+1}\setminus B(R))}.
\end{align*}
Both terms tend to $0$ as $R$ tends to infinity, thus $\int_{\H^{2m+1}}\xi\omega\wedge\beta=0$. 
\end{proof}

This proves Theorem \ref{vanish} except in degree $k=m$. The argument collapses in this case, since  primitives of left-invariant $(m+1)$-forms have at least quadratic growth.

\section{Averages on Heisenberg group: special case}
\label{poincare}

We now describe a symmetric argument: produce a primitive of the $L^1$ form $\omega$ with linear growth. It applies for all degrees but $m+1$, and so covers the special case $k=m$.

Since $\omega$ is not in $L^\infty$ but is $L^1$, linear growth needs be taken in the $L^1$ sense: the $L^1$ norm of the primitive in a shell of radius $R$ is $O(R)$. It is not necessary to produce a global primitive with this property. It is sufficient to produce such a primitive $\phi_R$ in the $R$-shell $B(\lambda R)\setminus B(R)$. Indeed, Stokes formula leads to an integral
$$
\int_{\H^{2n+1}}\xi\omega\wedge\beta=\pm\int_{B(\lambda R)\setminus B(R)}d_c\xi\wedge\phi\wedge\beta
$$
which does not depend on the choice of primitive $\phi$.

\subsection{$\ell^{q,1}$ cohomology of bounded geometry Riemannian and subRiemannian manifolds}
\label{discretization}

By definition, the $\ell^{q,p}$ cohomology of a bounded geometry Riemannian manifold is the $\ell^{q,p}$ cohomology of every bounded geometry simplicial complex quasiisometric to it. For instance, of a bounded geometry triangulation.

Combining results of \cite{BFP3} and Leray's acyclic covering theorem (in the form described in \cite{pansu2017cup}), one gets that for $q=\frac{n}{n-1}$, the $\ell^{q,1}$-cohomology of a bounded geometry Riemannian $n$-manifold $M$ is isomorphic to the quotient
$$
L^{q,1}H^\cdot(M)=L^1(M)\cap ker(d)/(L^1\cap dL^q(M))
$$
of closed forms in $L^1$ by differentials of forms in $L^q$. In particular, if $M$ is compact, for all $p\leq \frac{n}{n-1}$, $L^{p,1}H^\cdot(M)$ is isomorphic to the usual (topological) cohomology of $M$. 

Similarly, if $M$ is a bounded geometry contact subRiemannian manifold of dimension $2m+1$ (hence Hausdorff dimension $Q=2m+2$), for $q=Q/(Q-1)$ (respectively $q=Q/(Q-2)$ in degree $m+1$), the $\ell^{q,1}$ cohomology of $M$ is isomorphic to the quotient
$$
L_c^{q,1}H^\cdot(M)=L^1(M)\cap ker(d_c)/(L^1\cap d_c L^q(M))
$$
of $d_c$-closed Rumin forms by $d_c$'s of Rumin forms in $L^q$.

This applies in particular to Heisenberg groups $\mathbb{H}^{2m+1}$, and also to shells in Heisenberg groups, but with a loss on the width of the shell.

\subsection{$L^1$-Poincar\'e inequality in shell $B(\lambda )\setminus B(1)$}

\begin{lemma}\label{shell}
There exist radii $0<\mu<1<\lambda<\mu'$ such that every $d_c$-exact $L^1$ Rumin $k$-form $\omega$ on $B(\mu')\setminus B(\mu)$ admits a primitive $\phi$ on $B(\lambda)\setminus B(1)$ such that
\begin{align*}
\Vert\phi\Vert_{L^1(B(\lambda)\setminus B(1))}\le C\cdot\Vert \omega\Vert_{L^1(B(\mu')\setminus B(\mu))}.
\end{align*}
\end{lemma}

In Euclidean space, the analogous statement can be proved as follows. Up to a biLipschitz change of coordinates, one replaces shells with products $[0,1]\times S^{n-1}$. On such a product, a differential form writes $\omega=a_t+dt\wedge b_t$ where $a_t$ and $b_t$ are differential forms on $S^{n-1}$. $\omega$ is closed if and only if each $a_t$ is closed and 
$$
\frac{\partial a_t}{\partial t}=b_t.
$$

Given $r\in[0,1]$, define
$$
\phi_r=e_t +dt\wedge f_t\quad \mbox{ where }\quad e_t=\int_r^t b_s\,ds,\quad f_t=0.
$$
Then $d\phi_r=\omega-a_r$. Set
$$
\phi=\int_0^1 \phi_r \,dr, \quad \mbox{ so that }\quad d\phi=\omega-\bar\omega\quad \mbox{ where }\quad \bar\omega=\int_0^1 a_r\,dr.
$$
Note that each $a_r$, and hence $\bar\omega$, is an exact form on $S^{n-1}$. Since 
$$
\Vert\bar\omega\Vert_{L^1(S^{n-1})}\leq \Vert\omega\Vert_{L^1([0,1]\times S^{n-1})},
$$
according to Subsection \ref{discretization}, there exists a form $\bar\phi$ on $S^{n-1}$ such that $d\bar\phi=\bar\omega$ and
$$
\Vert\bar\phi\Vert_{L^1(S^{n-1})}\leq C\,\Vert\bar\omega\Vert_{L^1(S^{n-1})}.
$$
Hence $\phi-\bar\phi$ is the required primitive.

 The Heisenberg group case reduces to the Euclidean case thanks to a smoothing homotopy constructed in \cite{BFP3}. In fact, since $\phi$ merely needs be estimated in $L^1$ norm (and not in the sharp $L^q$ norm), only the first, elementary, steps of \cite{BFP3} are required, resulting in the following result.

\begin{lemma}\label{homotopyshell}
For every radii $\mu<1<\lambda<\mu'$, there exists a constant $C$ with the following property. For every $d_c$-exact $L^1$ Rumin form $\omega$ on the large shell $B(\mu')\setminus B(\mu)$ of $\mathbb{H}^{2m +1}$, there exist $L^1$ Rumin forms $T\omega$ and $S\omega$ on the smaller shell $B(\lambda)\setminus B(1)$ such that $\omega=d_c T\omega+S\omega$ on the smaller shell,
$$
\Vert T\omega\Vert_{L^1(B(\lambda)\setminus B(1))}+\Vert S\omega\Vert_{W^{1,1}(B(\lambda)\setminus B(1))} \leq C\,\Vert\omega\Vert_{L^1(B(\mu')\setminus B(\mu))}.
$$
Here, the $W^{1,1}$ norm refers to the $L^1$ norms of the first horizontal derivatives.
\end{lemma}
\begin{proof}
Pick a smooth function $\chi_1$ with compact support in the large shell $A$. According to Lemma 6.2 of \cite{BFP3}, there exists a left-invariant pseudodifferential operator $K$ such that the identity
$$
\chi_1=d_c K\chi_1+Kd_c\chi_1
$$
holds on the space of Rumin forms
$$
L^1\cap d_c^{-1}L^1:=\{\alpha\in L^1(A)\,;\,d_c\alpha\in L^1(A)\}.
$$
$K$ is the operator of convolution with a kernel $k$ of type $1$ (resp. $2$ in degree $n+1$). Using a cut-off, write $k=k_1+k_2$ where $k_1$ has support in an $\epsilon$-ball and $k_2$ is smooth. Since $k_1=O(r^{1-Q})$ or $O(r^{2-Q})\in L^1$, the operator $K_1$ of convolution with $k_1$ is bounded on $L^1$. Hence $T=K_1 \chi$ is bounded on $L^1$ forms defined on $A$. Whereas $S=d_c K_2\chi_1$ is bounded from $L^1$ to $W^{s,1}$ for every integer $s$. If $\mu'>\lambda+2\epsilon$ and $\mu<1-2\epsilon$, the multiplication of $\omega$ by $\chi_1$ has no effect on the restriction of $d_c K_1\omega$ or $K_1 d_c \omega$ to the smaller shell, hence, in restriction to the smaller shell, 
$$
d_c T\omega=d_c K_1\chi_1\omega=(d_c K_1+K_1 d_c) \omega.
$$
It follows that
$$
S\omega =d_c K_2\chi_1\omega=(d_c K_2+K_2 d_c) \omega,
$$
and finally, in restriction to the smaller shell, 
$$
\omega=d_c T\omega+S\omega.
$$
\end{proof}
\textbf{Proof of Lemma \ref{shell}}.

\begin{proof}
Using the exponential map, one can use simultaneously Heisenberg and Euclidean tools. Pick $\lambda,\mu,\mu'$ such that the medium Heisenberg shell $B(\mu'-2\epsilon)\setminus B(\mu+2\epsilon)$ contains the Euclidean shell $A_{eucl}=B_{eucl}(2)\setminus B_{eucl}(1)$, which in turn contains the smaller Heisenberg shell $B(\lambda)\setminus B(1)$. Apply Lemma \ref{homotopyshell} to a $d_c$-closed $L^1$ form $\omega$ defined on the larger Heisenberg shell. Up to $d_c T\omega$, and up to restricting to the medium shell, one can replace $\omega$ with $S\omega$ which has its first horizontal derivatives in $L^1$. Apply Rumin's homotopy $\Pi_E=1-dd_0^{-1}-d_0^{-1}d$ to get a usual $d$-closed differential form $\beta=\Pi_E S\omega$ belonging to $L^1$. Use the Euclidean version of Lemma \ref{shell} to get an $L^1$ primitive $\gamma$, $d\gamma=\beta$, on the Euclidean shell $A_{eucl}$. Apply the order zero homotopy $\Pi_{ E_0}=1-d_0 d_0^{-1}-d_0^{-1}d_0$ to get a Rumin form $\phi=\Pi_{ E_0}\gamma$. Its restriction to the smaller Heisenberg shell satisfies $d_c\phi=\omega$ and its $L^1$ norm is controlled by $\Vert\omega\Vert_1$.
\end{proof}
\subsection{$L^1$-Poincar\'e inequality in scaled shell $B(\lambda R)\setminus B(R)$}\label{Exact omega on the shell}

Let $0<\mu<1<\lambda<\mu'$. Let $\omega$ be a Rumin $k$-form on the scaled annulus $B(\mu' R)\setminus B(\mu R)$. Assume that there exists a Rumin $(k-1)$-form $\phi$ on the thinner shell on $B(\lambda R)\setminus B(R)$ such that $\omega=d_c\phi$ on that shell. 

Let's denote the dilation by $R$ as
\begin{align*}
\delta_R:B(\lambda)\setminus B(1)\to B(\lambda R)\setminus B(R)
\end{align*}
then we can consider the pull-back of both forms:
\begin{itemize}
\item $\omega_R:=\delta^\ast_R(\omega)$ on $B(\lambda)\setminus B(1)$, and
\item $\phi_R:=\delta^\ast_R(\phi)$ on $B(\lambda)\setminus B(1)$.
\end{itemize}

Since $\delta_R^\ast$ commutes with the Rumin differential $d_c$, we have
\begin{align*}
\omega_R=\delta_R^\ast(\omega)=\delta_R^\ast(d_c\phi)=d_c(\delta_R^\ast\phi)=d_c\phi_R\,.
\end{align*}
Then, for $\omega_R$ we have
\begin{align*}
\Vert\delta_R^\ast\omega\Vert_{L^1(B(\lambda)\setminus B(1))}=&\int_{B(\lambda)\setminus B(1)}\vert\omega(\delta_R(x))\vert\cdot R^wdx\underbrace{=}_{y=\delta_R(x)}R^w\int_{B(\lambda R)\setminus B(R)}\vert\omega\vert(y)\cdot\frac{1}{R^{Q-1}}dy\\=&R^{w-(Q-1)}\int_{B(\lambda R)\setminus B(R)}\vert\omega\vert(y)dy=R^{w-(Q-1)}\cdot
\Vert\omega\Vert_{L^1(B(\lambda R)\setminus B(R)}\,,
\end{align*}
so that
\begin{align}\label{omegaL1norm}
\Vert\omega_R\Vert_{L^1(B(\lambda)\setminus B(1))}=R^{w-(Q-1)}\Vert\omega\Vert_{L^1(B(\lambda R)\setminus B(R))}
\end{align}
where $w$ is the weight of the $k$-form $\omega$. 

Likewise, for the $(k-1)$-form $\phi$ we get
\begin{align}\label{phiL1norm}
\Vert\delta_R^\ast\phi\Vert_{L^1(B(\lambda)\setminus B(1))}=R^{\tilde{w}-(Q-1)}\Vert\phi\Vert_{L^1(B(\lambda R)\setminus B(R))}\,,
\end{align}
where in this case $\tilde{w}$ is the weight of the form $\phi$.

Since we are working on $\mathbb{H}^{2m+1}$ and $\omega=d_c\phi$ (and likewise $\omega_R=d_c\phi_R$), we have
\begin{itemize}
\item $\tilde{w}=w-1$, if $k\neq m+1$, and 
\item $\tilde{w}=w-2$, if $k=m+1$. 
\end{itemize}

According to Lemma \ref{shell}, one can find a $(k-1)$-form $\phi_R$ on $B(\lambda)\setminus B(1)$ such that
\begin{align*}
\Vert\phi_R\Vert_{L^1(B(\lambda)\setminus B(1))}\le C\cdot\Vert \omega_R\Vert_{L^1(B(\mu')\setminus B(\mu))}
\end{align*}
so, using the equalities (\ref{omegaL1norm}) and (\ref{phiL1norm}) we get the following inequality:
\begin{align*}
\Vert\phi\Vert_{L^1(B(\lambda R)\setminus B(R))}\le C\cdot R^{w-\tilde{w}}\Vert \omega\Vert_{L^1(B(\mu' R)\setminus B(\mu R))}
\end{align*}
which divides into the following two cases
\begin{itemize}
\item $\Vert\phi\Vert_{L^1(B(\lambda R)\setminus B(R))}\le C\cdot R\cdot\Vert \omega\Vert_{L^1(B(\mu' R)\setminus B(\mu R))}$ if $k\neq m+1$, and
\item $\Vert\phi\Vert_{L^1(B(\lambda R)\setminus B(R))}\le C\cdot R^2\cdot\Vert \omega\Vert_{L^1(B(\mu' R)\setminus B(\mu R))}$ if $k=m+1$.
\end{itemize}
Only the first case is useful for our purpose.

\subsection{Independence on the choice of primitive}

Let us consider the exact Rumin form $\omega\in E_0^{k}$ and let $\phi,\,\psi\in E_0^{k-1}$ be two primitives of $\omega$ on $B(\lambda R)\setminus B(R)$, i.e. $d_c\psi=d_c\phi=\omega$. Let $\beta$ be an arbitrary left-invariant Rumin form $\beta$ of complementary degree $2m+1-k$. 

If $k\not=2m$, then $H^{k}(B(\lambda R)\setminus B(R))=0$, which means that there exists a Rumin $(k-2)$-form $\alpha$ such that $d_c\alpha=\psi-\phi$. 

If $k\not=m+1$, the degree of $\beta$ is $2m+1-k\not=m$, so $d_c(\xi\beta)=(d_c\xi)\wedge\beta+\xi d_c\beta=(d_c\xi)\wedge\beta$, thus $\gamma=(d_c\xi)\wedge\beta$ is a well-defined $d_c$-closed Rumin form (this is a special case of Proposition \ref{dc of Wedge Product}).

If on top we also assume $k\not=m+2$, given  $\gamma=d_c(\xi\beta)=d_c\xi\wedge\beta$, we have that the form $\alpha\wedge\gamma$ has degree $2m\ge m+1$, so we can apply Proposition \ref{dc of Wedge Product} and obtain the following equality

\begin{align*}
d(\alpha\wedge\gamma)
=d_c(\alpha\wedge\gamma)\pm\alpha\wedge d_c\gamma
=(d_c\alpha)\wedge \gamma
=(\psi-\phi)\wedge (d_c\xi)\wedge\beta.
\end{align*}

Since by construction $d_c\xi$ has compact support in $B(\lambda R)\setminus B(R)$, 
\begin{align*}
\int_{\mathbb{H}^{2m+1}}d_c\xi\wedge(\psi-\phi)\wedge\beta=0.
\end{align*}
Therefore, when $k\not=m+1,m+2,2m$, we can replace a given primitive $\psi$ with any other arbitrary primitive $\phi$ of $\omega$ on the scaled shell $B(\lambda R)\setminus B(R)$.

\subsection{Vanishing of averages}

\begin{prop}
Given $\omega\in E_0^{k}$ an $L^{1}$, $d_c$-closed Rumin form in $\mathbb{H}^{2m+1}$, then the integral
\begin{align*}
\int_{\mathbb{H}^{2m+1}}\omega\wedge\beta
\end{align*}
vanishes for all left-invariant Rumin forms $\beta$ of complementary degree, $\beta\in E_0^{2m+1-k}$, provided $k\not=m+1,m+2,2m$.
\end{prop}

\begin{proof}
We assume that $k\neq m+1, m+2, 2m$.

Let $\psi$ be a global primitive of $\omega$ on $\mathbb{H}^{2m+1}$. Let us first analyse the following identities
\begin{align*}
\xi\omega\wedge\beta&=\xi(d_c\psi)\wedge\beta=-d_c\xi\wedge\psi\wedge\beta+d(\xi\psi\wedge\beta).
\end{align*}
Let $\phi$ be the primitive of $\omega$ on $B(\lambda R)\setminus B(R)$ introduced in Section \ref{Exact omega on the shell}. We can then replace $\int_{\mathbb{H}^{2m+1}}d_c\xi\wedge\psi\wedge\beta$ with $\int_{\mathbb{H}^{2m+1}}d_c\xi\wedge\phi\wedge\beta$, and by applying Stokes' theorem, we get
\begin{align*}
\bigg\vert\int_{\mathbb{H}^{2m+1}}\xi\omega\wedge\beta\bigg\vert
&=\bigg\vert\int_{B(\lambda R)\setminus B(R)}d_c\xi\wedge\psi\wedge\beta\bigg\vert\\
&=\bigg\vert\int_{B(\lambda R)\setminus B(R)}d_c\xi\wedge\phi\wedge\beta\bigg\vert\\ 
&\le\Vert d_c\xi\Vert_{\infty}\Vert\beta\Vert_{\infty}\Vert\phi\Vert_{L^{1}(B(\lambda R)\setminus B(R))}.
\end{align*}

Finally, knowing that $\Vert d_c\xi\Vert_\infty\le C'/R$ and applying the Poincar\'e inequality on the shell \begin{align*}
\Vert\phi\Vert_{L^1(B(\lambda R)\setminus B(R))}\le C\cdot R\cdot\Vert \omega\Vert_{L^1(B(\mu' R)\setminus B(\mu R))}
\end{align*}
found in Subsection \ref{Exact omega on the shell}, we finally get
\begin{align*}
    \bigg\vert\int_{\mathbb{H}^{2n+1}}\xi\omega\wedge\beta\bigg\vert
&\le CC' \Vert\omega\Vert_{L^1(B(\mu' R)\setminus B(\mu R))}\,.
\end{align*}

Using the cut-off function $\xi$ introduced in Definition \ref{cutoff defin}, we have
\begin{align*}
\int_{\mathbb{H}^{2m+1}}\xi\omega\wedge\beta=\int_{B(R)}\omega\wedge\beta+\int_{B(\lambda R)\setminus B(R)}\xi\omega\wedge\beta\,.
\end{align*}

Hence
\begin{align*}
\bigg\vert\int_{B(R)}\omega\wedge\beta\bigg\vert&\le\bigg\vert\int_{\mathbb{H}^{2n+1}}\xi\omega\wedge\beta\bigg\vert+\bigg\vert\int_{B(\lambda R)\setminus B(R)}\xi\omega\wedge\beta\bigg\vert\\ &\le CC'\Vert\omega\Vert_{B(\mu' R)\setminus B(\mu R)}+\int_{B(\lambda R)\setminus B(R)}\Vert\beta\Vert_\infty\cdot\vert \omega\vert\\
&\le {C''}\Vert\omega\Vert_{L^1(B(\mu' R)\setminus B(\mu R))}.
\end{align*}
Since $\Vert\omega\Vert_{L^1(B(\mu' R)\setminus B(\mu R))}\to 0$ as $R\to\infty$, we get our result
\begin{align*}
\int_{\mathbb{H}^{2m+1}}\omega\wedge\beta=0\,.
\end{align*}

\end{proof}

This completes the proof of Theorem \ref{vanish}.

\begin{oss}
Let us notice that this method would not work in the case where $k=m+1$, since we would obtain the following inequality
\begin{align*}
\int_{B(R)}\omega\wedge\beta\le C\cdot R\lVert\omega\Vert_{L^1(B(\lambda R)\setminus B(R))}
\end{align*}
which is not conclusive.
\end{oss}

\section*{Acknowledgments}

P.P. is supported by Agence Nationale de la Recherche, ANR-15-CE40-0018 SRGI. F.T. is partially supported by the Academy of Finland grant
288501 and by the
ERC Starting Grant 713998 GeoMeG.

\bibliographystyle{plain}
\bibliography{bibli.bib}

\newcommand{\noop}[1]{}
\begin{thebibliography}{1}

\bibitem{BFP2}
A.~Baldi, B.~Franchi, and P.~Pansu.
\newblock Poincar\'e and {S}obolev inequalities for differential forms in
  {H}eisenberg groups.
\newblock {\em arXiv preprint arXiv:1711.09786}, 2017.

\bibitem{BFP3}
A.~Baldi, B.~Franchi, and P.~Pansu.
\newblock $ {L}^{1}$-{P}oincar\'e inequalities for differential forms on
  {E}uclidean spaces and {H}eisenberg groups.
\newblock {\em arXiv preprint arXiv:1902.04819}, 2019.

\bibitem{Bourgain2007}
J.~Bourgain and H.~Brezis.
\newblock New estimates for elliptic equations and hodge type systems.
\newblock {\em Journal of the European Mathematical Society}, 009(2):277--315,
  2007.

\bibitem{chanillo2009subelliptic}
S.~Chanillo and J.~van Schaftingen.
\newblock Subelliptic {B}ourgain-{B}rezis estimates on groups.
\newblock {\em Mathematical Research Letters}, 16(3):487--501, 2009.

\bibitem{pansu2017cup}
P.~Pansu.
\newblock Cup-products in ${L}^{q,p}$-cohomology: discretization and
  quasi-isometry invariance.
\newblock {\em arXiv preprint arXiv:1702.04984}, 2017.

\bibitem{Pansu-Rumin}
P.~Pansu and M.~Rumin.
\newblock On the $\ell^{q,p}$ cohomology of {C}arnot groups.
\newblock {\em Annales H. Lebesgue}, \noop{3001} to appear.

\bibitem{Troyanov}
S.~Pigola, A.~G. Setti, and M.~Troyanov.
\newblock The connectivity at infinity of a manifold and $l^{q,p}$-{S}obolev
  inequalities.
\newblock {\em Expositiones Mathematicae}, 32(4):365--383, 2014.

\bibitem{rumin1994}
M.~Rumin.
\newblock Differential forms on contact manifolds.
\newblock {\em J. Differ. Geom}, 39(2):281--330, 1994.

\end{thebibliography}

\bigskip
\vspace{0.5cm}
\tiny{
\noindent
Pierre Pansu 
\par\noindent Laboratoire de Math\'ematiques d'Orsay,
\par\noindent Universit\'e Paris-Sud, CNRS,
\par\noindent Universit\'e
Paris-Saclay, 91405 Orsay, France.
\par\noindent
e-mail: pierre.pansu@math.u-psud.fr\newline
}

\tiny{
\par\noindent
Francesca Tripaldi 
\par\noindent Department of Mathematics and Statistics,
\par\noindent University of Jyv\"askyl\"a,
\par\noindent 40014, Jyv\"askyl\"a, Finland.
\par\noindent email: francesca.f.tripaldi@jyu.fi
}

\end{document}